\documentclass[reqno,12pt,a4paper]{amsart}
\usepackage{amsmath,amssymb,exscale, amscd}
\usepackage[utf8]{inputenc}
\usepackage[T1]{fontenc}
\usepackage{charter}
\usepackage{amsmath}
\usepackage[all]{xy}
\usepackage[english]{babel}
\usepackage{mathrsfs}
\usepackage{mathalfa}
\usepackage{graphicx}
\usepackage{tikz}
\usepackage{epigraph}
\usepackage{setspace}
\usepackage{enumitem}
\usepackage[utf8]{inputenc}
\usepackage{float}
\usepackage{amsthm}
\usepackage{amssymb}
\usepackage{a4wide}
\usepackage{mathtools}
\usepackage{verbatim}

\DeclarePairedDelimiter\floor{\lfloor}{\rfloor}
\usepackage{hyperref}
                 \hypersetup{ pdfborder={0 0 0}, 
                              colorlinks=true, 
                              citecolor=blue,
                              linktoc=page,
                              pdfauthor={}, 
                              pdftitle={}}
\renewcommand{\ref}{\hyperref}
\newtheorem{theorem}{Theorem}[section]
\newtheorem{lemma}[theorem]{Lemma}
\newtheorem{corollary}[theorem]{Corollary}
\newtheorem{proposition}[theorem]{Proposition}
\theoremstyle{definition}
\newtheorem{definition}[theorem]{Definition}
\newtheorem{remark}[theorem]{Remark}
\newtheorem{example}[theorem]{Example}

\theoremstyle{plain}
\DeclareMathOperator{\res}{Res}
\DeclareMathOperator{\ver}{Ver}
\DeclareMathOperator{\spec}{Spec}
\DeclareMathOperator{\fract}{Frac}

\makeatletter
\newcommand{\extp}{\@ifnextchar^\@extp{\@extp^{\,}}}
\def\@extp^#1{\mathop{\bigwedge\nolimits^{\!#1}}}
\makeatother
\font\smallrm=cmr7

\author[
\smallrm{W.~Sousa}
]
{W\'allace Sousa}
\title[Discriminant, vertex and general type of a family of polynomials]{Discriminant, vertex and general type of a family of polynomials}

\keywords{Discriminants; family of polynomials; general type; vertex}

\begin{document}

\begin{abstract} In this paper we give a formal expression to the limits of dual plane curves by using the method introduced by Katz. We use this formal expression to give a formula to compute the \textit{vertices} of $C(0)$ a plane curve in $C(t)$ a one-parameter family of plane curves and we show that those vertices, or equivallently the limit of dual plane curves, depend only on the first $\tau$ terms of that family, where $\tau$ is the \textit{general type} of $C(t)$.  \end{abstract}

\maketitle

\section{Introduction} 

\subsection*{The limits of dual plane curves} Let $k$ be an algebraically closed field of characteristic zero, $k[[t]]$ the ring of formal power series and $k((t)):=k[[t]][1/t]$ the field of formal Laurent series and $S:=k[x,y,z]$. Consider $C(t):=Z(F(t)) \subset \mathbb{P}_{k[[t]]}^2$ the family of curves of degree $n$, that is, the family of curves cut out by
$$
F(t)=F_0 +F_1t+F_2t^2 + O(t^3) \in S[[t]],
$$
a homogeneous power series of degree $n>0$, where $F_0\not = 0$ and $O(t^r)$ indicate higher other terms. We may change to a general coordenate system with the following properties:
\begin{enumerate}
\item[1)] $F_0(1,0,0)\not = 0$;
\item[2)] $z=0$ is not tangent to $F_0$.
\end{enumerate}
Now, consider $C^\ast \subset \mathbb{P}_{k((t))}^ 2$ the generic fiber of this family, that is, the plane curve cut out by
$$
F(t)^\ast=F_0 +F_1t+F_2t^2 + O(t^3) \in S((t)).
$$
Supposing that $C^\ast$ is smooth, consider $(C^\ast)^\vee \subset (\mathbb{P}_{k((t))}^2)^\vee$ the usual dual curve of $C^\ast$, where $(\mathbb{P}^2_{k((t))})^\vee$ is the space of lines in $\mathbb{P}_{k((t))}^{2}$. We may ask which plane curve the dual curve to $C^\ast$ degenerates to as $t$ aproches $0$. In other words, what is $\lim_{t \to 0}(C^\ast)^\vee $, the \textit{limit of} $(C^\ast)^\vee$, of the family? This problem is not new, in fact it is surfacing in works by Maillard~\cite{Ma} and Zeuthen~\cite{Z1}, \cite{Z2}.

To compute the limit of the dual plane curves, that is, the schematic boundary of $(C^\ast)^\vee$ in $(\mathbb{P}^2_k)^\vee$ we consider to calculate 
$$
L\cap \lim_{t \to 0}(C^\ast)^\vee
$$
for a general line $L \subset (\mathbb{P}^2_k)^\vee$. Without loss of generality, we may assume that $L$ is the line dual to the point $(1:0:0) \in \mathbb{P}_k^2$. 

 Write $F_0:=\prod_{\nu=1}^mF_{\nu,0}^{n_\nu}$ where the $F_{\nu,0}$ are the irreducible components. So if $F_{\nu,0}$ has degree $d_\nu$, we can by assumption write an equation for $F_{\nu,0}$ as $x^{d_\nu} +$ other terms. Restrict to the affine open $z=1$, we view $y$ as an affine coordinate on $L$, with $y=y_0$ corresponding to the line $\ell_{y_0} :=Z(y-y_0z) \in L$. Also, we may view $F(x)=F(x,y,1,t) \in A[[t]][x]$, where $A:=k[y]$, as a monic polynomial with coefficients in $A[[t]]$ and we consider $\varphi:\spec A[[t]][x] / (F) \rightarrow \spec A[[t]]$ the corresponding morphism.

The discriminant scheme of $\varphi$ is $Z(\Delta(F))$ the zero locus of the discriminant of $F$ (\cite{EH}, Corollary V-23, p.~234) which is flat over $\spec k[[t]]$ if, and only if, $\Delta(F)$ is not divisible by $t$. Write $\Delta(F)=t^\alpha\Delta'$, where $\Delta' \in A[[t]]$ is not divisible by $t$. Thus $Z(\Delta')$ is a subscheme of $ \spec A[[t]]$, flat over $\spec k[[t]]$ and it agrees with $Z(\Delta(F))$ over the generic point of $\spec k[[t]]$. The \textit{special fiber} of $Z(\Delta')$ is clearly defined by $\Delta'_0$ the \textit{leading coefficient} of $\Delta'$. By assumption, $Z(\Delta')$ gives the values $y_0$ of $y$ for which $l_{y_0}$ is tangent to $C^\ast$; so we may view $ L\cap \lim_{t \to 0}(C^\ast)^\vee$ as $ Z(\Delta'_0)$. 

Assume $F_0(x,y,1)=\prod_\nu f_\nu^{n_\nu}$, where $f_\nu:=F_{\nu,0}|_{z=1}$ with $F_{\nu,0}$ the irreducible components of $F_0$ as before. Write $f_\nu=\prod_\ell(x-a_{\nu \ell})$, with $a_{\nu \ell}$ the formal roots in some extension of $\fract(A)$. With this notation Katz proved the following.
\begin{theorem}[\cite{Ka}, Theorem~3]\label{Theorem 1} Assume $F$ regular. Then $$\Delta'_0=\prod_\nu\Delta(f_\nu)^{n_\nu}\prod_{\nu,\ell}\Delta_{f_{\nu}}(F)(a_{\nu \ell})\prod \res(f_\nu,f_{\nu'})^{2(n_\nu+n_{\nu'}-n_\nu n_{\nu'})}.$$ 
\end{theorem}
If $C\subset \mathbb{P}^2_k$ is an integral curve, then $C^\vee$  the dual curve of $C$ is made up of the usual dual curve and the lines dual to the singular points $P \in C$, with some multiplicity (\cite{EH}, p.~240). If $C,D$ are curves without common components, then we denote by $(C\cdot D)^\vee$ the union of the pencils of lines through the points $C \cap D$, each counted with the same multiplicity that the corresponding point has in the cycle $C \cdot D$.

Going back to the regular families case and defining $C_\nu:=Z(F_{\nu,0})$, we may identify $Z\big(\prod_{\ell}\Delta_{f_\nu}(F)(a_{\nu \ell})\big)$ with $L\cap \big(Z(\Delta_{f_{\nu}}(F))\cdot C_\nu\big)^\vee$, $Z\big(\Delta(f_\nu)\big)$ with $L\cap C_\nu^\vee$ and $Z\big(\res(f_\nu,f_{\nu'})\big)$ with $L\cap (C_\nu\cdot C_{\nu'})^\vee$. But the line $L \subset (\mathbb{P}_k^{2})^\vee$ was a general line by choice of coordinates, so we may interpret Theorem~\ref{Theorem 1} as:

\hspace{-0.4cm}Assume $F$ is regular. Then
$$
\lim_{t \to 0}(C^\ast)^\vee=  \sum_\nu n_\nu C_\nu^\vee + \sum_{\nu}\Big(Z\big(\Delta_{F_{\nu,0}}(F)\big)\cdot C_\nu\Big)^\vee + \sum 2(n_\nu+n_{\nu'}-n_\nu n_{\nu'})(C_\nu\cdot C_{\nu'})^\vee.
$$

\hspace{-0.4cm}Also Katz predicted that (\cite{Ka}, Remark~(3.7))

\begin{center} \textit{The techniques described here can be used to find limits of dual curves of nonregular degenerations. One merely computes more terms in the formal Puiseux expansions (...) The result has the same form as in Theorem: a contribution from the dual of each component, a ``generalized formal part'' and a contribution from each intersection.}
\end{center}   
So it is natural to ask: \begin{enumerate}
	\item[$(i)$] How to determine the ``generalized formal part''?
	\item[$(ii)$] How many terms in the formal Puiseux expasions do we need? Or equivalently, how many terms of $F(t)$ do we need to consider in order to compute $\lim_{t\to 0}(C^\ast)^\vee$?
\end{enumerate}
Our goal is to answer the above questions. To answer the question $(i)$ we consider $\Delta(F,f_\nu)(x,t)$ the \textit{generalized formal discriminant} of $F$ with respect to $f_\nu$ (Definition~\ref{Definition 2.3}).  The generalized formal discriminant is what we need to proof the following.

\begin{theorem}[Theorem~\ref{Theorem 2.7}] \label{Theorem 1.1} \textit{Suppose $\Delta'\not=0$ and $ F_0=\prod_\nu f_\nu^{n_\nu}$, where $f_\nu=\prod_\ell(x-a_{\nu \ell})$ for all $\nu$ with $a_{\nu \ell}\not=a_{\nu'\ell'}$ if $(\nu,\ell)\not=(\nu',\ell')$. Then}
\begin{equation}\label{(0.2.1)}
\Delta'_0=\prod_\nu\Delta(f_\nu)^{n_\nu}\prod_{\nu,\ell}\Delta(F,f_\nu)(a_{\nu \ell},t)_0\prod \res(f_\nu,f_{\nu'})^{2(n_\nu+n_{\nu'}-n_\nu n_{\nu'})}.
\end{equation}
\end{theorem}
\hspace{-0.4cm}As we may see in Remark~\ref{Remark 2.6}, if $F$ is a regular family then   
$$
\prod_\ell\Delta(F,f_\nu)(a_{\nu \ell},t)_0=\prod_\ell\Delta_{f_{\nu}}(F)(a_{\nu \ell}),
$$
where $\Delta_{f_{\nu}}(F)$ is the formal discriminant of $F$ with respect to $f_{\nu}$ defined by Katz. In this case, we may rewrite the Equation~\eqref{(0.2.1)} as
$$
\Delta'_0=\prod_\nu\Delta(f_\nu)^{n_\nu}\prod_{\nu,\ell}\Delta_{f_{\nu}}(F)(a_{\nu \ell})\prod \res(f_\nu,f_{\nu'})^{2(n_\nu+n_{\nu'}-n_\nu n_{\nu'})}.
$$
As a consequence of the Theorem~\ref{Theorem 1.1} we may use the Corollary~\ref{Corollary 3.8} and similar identifications as in the regular families case to obtain the next result.

\begin{corollary}\label{Corollary 1.3} Suppose $F(t)$ satisfies $(A_\nu,\Delta_{\nu,n(\nu)})$-Zd for each $f_\nu$ with $n_\nu>1$, that is, suppose that for each $n_\nu>1$ there exists a decomposition 
$$A_\nu F(t)=(F_{\nu,0} B_{\nu,0}+B_{\nu,1}t^1+\cdots + B_{\nu,n(\nu)}t^{n(\nu)-1})^{n_\nu}+\Delta_{\nu,n(\nu)}t^{n(\nu)}+O(t^{n(\nu)+1}), $$
where $A_l,B_{i,j},\Delta_{r,s} \in S$ are homogeneous polinomials such that $\gcd(A_\nu\Delta_{\nu,n(\nu)},F_{\nu,0})=1$. Then 
\begin{align*} \lim_{t \to 0}(C^\ast)^\vee =  \sum_\nu n_\nu C_\nu ^\vee & + \sum_{\nu}\Big[\sum_{\nu'\not = \nu}n_{\nu'}\big(C_\nu\cdot C_{\nu'}\big)^\vee\Big] \\ &  +  \sum_{\nu}(n_\nu-1)\Big[\Big(Z\big(\Delta_{\nu,n(\nu)}\big)\cdot C_\nu\Big)^\vee - \Big(Z(A_\nu)\cdot C_\nu\Big)^\vee\Big].
\end{align*}
\end{corollary} 
It is not hard to see that if a family $F(t)$ is a \textit{degeneration along to a quasi-general direction} or if $F(t)$ is a \textit{Zeuthen family} then $F(t)$ satisfies the hypothesis of Corollary~\ref{Corollary 1.3}. Thus Corollary~\ref{Corollary 1.3} generalizes (\cite{ENW}, Corollary~6.2 and Corollary 7.6).

\subsection*{The vertices of a plane curve} We say that a point $P\in C(0):=Z(F_0)$ is a \textit{vertex} of $C(0)$ in $C(t)$ if $P^\vee$ is a component of $\lim_{t \to 0}(C^\ast)^\vee$. By Theorem~\ref{Theorem 1.1} we conclude that $\Delta'_0$ depends on the contribution of the dual of each component of the special fiber and the another part which we call it by \textit{vertex} of the family $F(t)$ and denote it by $\ver(F)$ (see Definition~\ref{Definition 3.1}). The vertex is a formal expression to the vertices of $C(0)$ in $C(t)$.


\begin{example}[\cite{Z1}, ``first kind''] Let  $C(t)$ be the family of curves of degree $n$ cut out by
$$F(t)=x^2A_{n-2} + B_nt+ O(t^2) \in S[[t]],$$
where $x$ does not divide $A_{n-2}$ nor $B_n$ and the plane curve cut out by $A_n$ is smooth. By Proposition~\ref{Proposition 3.2} we have that
$$\ver(F)=\res(x,-4B_nA_{n-2}^3).$$ The vertices of $C(0)$ in $C(t)$ are given by $$\big(Z(x)\cdot Z(B_n A_{n-2}^3)\big).$$
\end{example}

\begin{example}[\cite{Z1}, ``second kind''] Let $C(t)$ be the family of curves of degree $n$ cut out by
$$
F(t)=x^2A_{n-2} + 2xB_{n-1}t + C_nt^2 + O(t^3) \in S[[t]],
$$
where $x$ does not divide $A_{n-2}(B_{n-1}^2-A_{n-2}C_n)$ and the plane curve cut out by $A$ is smooth. By Proposition~\ref{Proposition 3.3} we have that
$$
\ver(F)=\res\Big(x,4(B_{n-1}^2-A_{n-2}C_n)A_{n-2}^2\Big).
$$
The vertices of $C(0)$ in $C(t)$ are given by $$\Big(Z(x)\cdot Z\big((B_{n-1}^2-A_{n-2}C_n)A_{n-2}^2\big)\Big).$$
\end{example}
 
\begin{example}[\cite{Z1}, ``third kind''] Let $C(t)$ be the family of curves of degree $n$ cut out by
$$
F(t)=x^2A_0+x^3A_1+2xB_0t+x^2B_1t+C_0t^2+xC_1t^2+D_1t^3+O(t^4) \in S[[t]],
$$
where $\res(x,B_0^2-A_0C_0)=0$, $x$ does not divide $A_{0}(-A_1B_0^3 + B_1B_0^2A_0 - C_1B_0A_0^2 + D_1A_0^3)$ and the plane curve cut out by $xA_1+A_0$ is smooth. By Proposition~\ref{Proposition 3.4} we have that
$$
\ver(F)=\res\Big(x,-4\big(-A_1B_0^3 + B_1B_0^2A_0 - C_1B_0A_0^2 + D_1A_0^3\big)\Big).
$$
The vertices of $C(0)$ in $C(t)$ are given by $$\Big(Z(x)\cdot Z\big(-A_1B_0^3 + B_1B_0^2A_0 - C_1B_0A_0^2 + D_1A_0^3\big) \Big).$$
\end{example}

\subsection*{The general type of a family of polynomials} We may note that by Theorem~\ref{Theorem 1.1} the question $(ii)$ is equivalent to the following one: \begin{enumerate}
	\item[(ii)'] How many terms of $F(t)$ do we need to consider in order to compute $\ver(F)$? 
\end{enumerate}
Now, by the above examples we have that $\ver(F)$ the vertex of the families of plane curves cut out by $F(t)=\sum_{i\geq 0} F_it^i$ a family of polynomials  of the \textit{first kind}, \textit{second kind}, or \textit{third kind} depends only on $(F_0,F_1)$, $(F_0,F_1,F_2)$ or $(F_0,F_1,F_2,F_3)$, respectively. Thus we have an answer for the question $(ii)'$ in these particular cases. In a general context, that is, if $C(t)$ is a general family of plane curves cut out by $F(t)$ a homogeneous power series then the \textit{general type} (Definition~\ref{Definition 4.8}) is the answer for the question $(ii)'$ as we may see in Theorem~\ref{Theorem 4.7}.

\subsection*{Setup} In Section 2 we give a formal expression to $\Delta'_0$ in terms of the formal general discriminant. In Section 3 we derive from this formal expression a formula to the vertex of a family of polynomials. We calculate the vertex of families of the first, second and third kinds in Zeuthen terminology. In Section 4 we show that the general type of a family of polynomials indicates how many terms of this family we need to consider in order to compute it's vertex. In addition, we calculate the general type of some families of polynomials.

\section{Generalized formal discriminant of a family of polynomials}

Let $k$ be an algebraically closed field of characteristic zero and $A$ a $k$-algebra and integral domain. A \textit{family of polynomials} of degree $n'$ with coefficients in $A$ is an expression
$$
F(x,t)=\sum_{i,j \geq 0} a_{i,j}t^ix^j=\sum_{i \geq 0} p_i(x)t^i,
$$
with $p_i(x) \in A[x]$, $\deg(p_0)=n'$ and $\deg(p_i(x))\leq n'$, for $i >0$. Suppose that $x^2 \mid p_0(x)$ but $x^2 \nmid F(t)$. Consider $P$ the Newton polygon of $F(t)$, that is, the convex hull of $\bigcup_{a_{i,j}\not = 0}[(i,j)+\mathbb{R}^2_{\geq 0}]$. Consider $d$, the largest integer such that $x^{d} \mid p_0(x)$. Denote the \textit{slope} of the \textit{leading edge} of $P$ through $(0,d)$ by $-1/m_1$. Now \textit{truncate} $F(t)$ by including  only the terms of $F(t)$ corresponding to the lattice on that leading edge to obtain
$$
F_1'=\sum a_{m_1(d-j),j } t^{m_1(d-j)}x^j,
$$
the \textit{first truncation} of $F(t)$. Finally, define $P_{F_1'}(s):=\sum a_{m_1(d-j),j}s^j \in A[s].$ Let $\alpha_{1,1},...,\alpha_{1,d}$ be the formal roots of $P_{F_1'}(s)$. The result is the formal \textit{Puiseux expansions}
$$
x=s_i=t^{m_{1,i}}(\alpha_{1,i} + t^{m_{2,i}}(\alpha_{2,i} + \cdots
$$
which we will see as $t$-roots of $F$, with $m_{1,i}:=m_1$ and $i=1,...,d$. Now, write 
$$
p_0(x)=a\prod_{i=1}^{n'-d}(x-a_i), \hspace{0.5cm} a_i\not = 0.
$$
The \textit{branches} of $F(t)$ are
$$
x=s_i=t^{m_{1,i}}(\alpha_{1,i} + t^{m_{2,i}}(\alpha_{2,i} + \cdots \hspace{0.5cm} (1\leq i \leq d)
$$
$$
x=s_{d+i}=a_i+t^{m_{1,d+i}}(\alpha_{1,d+i} + t^{m_{2,d+i}}(\alpha_{2,d+i} + \cdots \hspace{0.5cm} (1\leq i \leq n'-d).
$$

\begin{definition}\label{Definition 2.1} Let $f \in A[[t]]$ be non-zero. The \textit{limit of} $f$ as $t\rightarrow 0$, denoted by $f_0$, is the leading coefficient of $f$ as a formal power series in $t$.
\end{definition}

\begin{remark}\label{Remaark 2.2} Let $ h(x)= a_nx^n+\cdots + a_0 \in A[x]$ be a degree $n$ polynomial ($a_n\not =0$). Let $\alpha_1,...,\alpha_n$ be the formal roots of $h(x)$, that is, $h(x)=a\prod_{i=1}^n(x-\alpha_i)$. The \textit{discriminant} of $h$ is given by
$$\Delta_A\big(h(x)\big)=a_n^{2n-2}\prod_{i<j}(\alpha_i-\alpha_j)^2.$$
If $g(x)=b_m\prod_{j=1}^m(x-\beta_j)$, then the \textit{resultant} of $h$ and $g$ is given by $$\res (h,g)=a_n^mb_m^n\prod_{i,j}(\alpha_i-\beta_j)=a_n^m\prod_ig(\alpha_i).$$
\end{remark}

Let $F(x,t)=\sum_{i \geq 0} F_i(x)t^i$ be a family of polynomials of degree $n'$ with coefficients in $A$ and let $F_0=\prod_\nu f_\nu(x)^{n_\nu}$, where $f_\nu$ are coprime and monic polynomials. Suppose $s_1(t),...,s_{n'}(t)$ be the $t$-roots of $F(t)$ as before and fixe $\nu$. Then we may write 
$$
F=\prod_\ell(x-s_\ell)\prod_{i=1}^{n_\nu}\big(f_\nu - s_i'(x,t)\big),
$$
where $f_\nu\big((s_\ell)_0\big)\not = 0$. In this case, we define $$ F_{(f_\nu)}(w):=\prod_\ell(x-s_\ell)\prod_{i=1}^{n_\nu}\big[w - s_i'(x,t)\big] \in B[w],$$
where $B:=\overline{K}[[t]][x]$, with $K:=\fract{A}$ and $\overline{K}$ is the algebraic closure of $K$.

\begin{definition}\label{Definition 2.3} The \textit{generalized formal discriminant} of $F(x,t)$ with respect to $f_\nu(x)$ is defined by
$$
\Delta(F,f_\nu)(x,t):=\Delta_B\big(F_{(f_\nu)}(w)\big)=\Big[\prod_\ell(x-s_\ell)\Big]^{2n_\nu-2}\prod_{i<j}\big[s_{i}'(x,t)-s_{j}'(x,t)\big]^2 .
$$
\end{definition}

\begin{lemma}\label{Lemma 2.4} Suppose $\Delta_{A[[t]]}(F)\not=0$ and $ F_0=\prod_\nu(x-a_{\nu})^{n_\nu}$, with $a_{\nu}\not=a_{\nu'}$ if $\nu\not = \nu'$ and let
$$
s_i=a_\nu + \cdots \hspace{0.25cm} \textrm{for }i=1,...,n_\nu \hspace{0.5cm} \textrm{ and } \hspace{0.5cm} s_{n_{\nu}+1},...,s_{n'}
$$
be the roots of $F$. The leading coefficient of $\Delta(F,x-a_v)(a_v,t)$ satisfies the following equation
$$
\Delta(F,x-a_{\nu})(a_{\nu},t)_0 = \Big[\prod_{ \nu'\not = \nu}(a_\nu-a_{\nu'}) ^{n_{\nu'}}\Big]^{2n_\nu-2}\prod_{1\leq i < j \leq n_\nu}(s_i-s_j)_0^2
$$
\end{lemma}

\begin{proof} Note that
$$
F_{(x-a_\nu)}(w)= \prod_{\ell> n_\nu}(x-s_\ell)\prod_{i=1}^{n_\nu}\big[w-(s_i-a_\nu)\big].
$$
By Definition~\ref{Definition 2.3} we have
$$
\Delta(F,x-a_\nu)(x,t)=\Big[\prod_{\ell>n_\nu}(x-s_\ell)\Big]^{2n_\nu-2}\prod_{1\leq i<j\leq n_\nu}(s_i-s_j)^2.
$$
The last equation implies 
$$
\Delta(F,x-a_\nu)(a_\nu,t)=\Big[\prod_{\ell>n_\nu}(a_\nu-s_\ell)\Big]^{2n_\nu-2}\prod_{1\leq i<j\leq n_\nu}(s_i-s_j)^2
$$
and
$$
\Delta(F,x-a_\nu)(a_\nu,t)_0=\Big[\prod_{\nu'\not =\nu}(a_\nu-a_{\nu'})^{n_{\nu'}}\Big]^{2n_\nu-2}\prod_{1\leq i<j\leq n_\nu}(s_i-s_j)_0^2,
$$
as desired.
\end{proof}

\begin{proposition}\label{Proposition 2.5} Suppose $\Delta_{A[[t]]}(F)\not=0$ and $F_0=\prod_\nu f_\nu^{n_\nu}$, where $f_\nu=\prod_\ell(x-a_{\nu \ell})$ with all $a_{\nu \ell}$ distincts. Then
$$
\Delta(F,x-a_{\nu \ell})(a_{\nu \ell},t)_0=\prod_{\ell' \not = \ell}(a_{\nu \ell}-a_{\nu \ell'})^{n_\nu(n_{\nu}-1)}\Delta(F,f_\nu)(a_{\nu \ell},t)_0.
$$
\end{proposition}
\begin{proof} Let
$$
s_{\nu \ell i_\nu}= a_{\nu \ell} + \cdots, \ i_\nu=1,...,n_\nu
$$
be the roots of $F$. So
\begin{align*}
F	= & \prod_{\nu'\not = \nu}\big[x - s_{\nu'\ell'i_{\nu'}}(t)\big] \prod_{i_\nu=1}^{n_\nu}\Big[\prod_\ell\big[x - s_{\nu \ell i_\nu}(t)\big]\Big]\\
	=& \prod_{\nu'\not = \nu}\big[x - s_{\nu'\ell'i_{\nu'}}(t)\big] \prod_{i_\nu=1}^{n_\nu}\Big[\prod_\ell\big[x-a_{\nu \ell} - \big(s_{\nu \ell i_\nu}(t)-a_{\nu \ell}\big)\big]\Big]\\
	=& \prod_{\nu'\not = \nu}\big[x - s_{\nu'\ell'i_{\nu'}}(t)\big] \prod_{i_\nu=1}^{n_\nu}\Big[f_\nu - \sum_\ell\big[\prod_{\ell'\not=\ell}(x-a_{\nu \ell'})(s_{\nu \ell i_\nu}(t)-a_{\nu \ell})\big] + \cdots\Big].
\end{align*}
By Definition~\ref{Definition 2.3} we have
\begin{align*}
\Delta(F,f_\nu)(a_{\nu \ell},t)		= & \Big[\prod_{\nu'\not = \nu}\big(a_{\nu \ell} - s_{\nu'\ell'i_{\nu'}}(t) \big) \Big]^{2n_\nu-2} \prod_{i_\nu\not = j_\nu}\Big[\prod_{\ell'\not = \ell}(a_{\nu \ell} - a_{\nu \ell'})(s_{\nu \ell i_\nu} - s_{\nu \ell j_\nu}) \Big] \\ 
									= & \Big[\prod_{\nu' \not= \nu}(a_{\nu \ell}-
a_{\nu' \ell'})^{n_{\nu'}} \Big]^{2n_\nu-2} \Big[\prod_{\ell'\not = \ell}(a_{\nu \ell} - a_{\nu \ell'}) \Big]^{n_\nu(n_\nu-1)} \prod_{i_\nu<j_\nu}(s_{\nu \ell i_\nu}-s_{\nu \ell j_\nu})^2
\end{align*}
and consequently
$$
\Delta(F,f_\nu)(a_{\nu \ell},t)_0= \Big[\prod_{\nu' \not= \nu}(a_{\nu \ell}-
a_{\nu' \ell'})^{n_{\nu'}} \Big]^{2n_\nu-2} \Big[\prod_{\ell'\not = \ell}(a_{\nu \ell} - a_{\nu \ell'}) \Big]^{n_\nu(n_\nu-1)}$$
\begin{equation}\label{(1.5.2)}
	\prod_{i_\nu<j_\nu}(s_{\nu \ell i_\nu}-s_{\nu \ell j_\nu})_0^2.
\end{equation}
Now, the desired result follows by using Lemma~\ref{Lemma 2.4}.
\end{proof}

\begin{remark}\label{Remark 2.6} Note that by the proof of (\cite{Ka}, Lemma~2) and by Lemma~\ref{Lemma 2.4} we have that if $F$ is \textit{regular} then $\Delta_{(x-a_{\nu \ell})}(F)(a_{\nu \ell})$ the \textit{formal discriminant} of $F$ with respect to $x-a_{\nu \ell}$ satisfies the following expression \begin{equation}\label{(1.6.2)}
	\Delta_{(x-a_{\nu \ell})}(F)(a_{\nu \ell})=\Delta(F,x-a_{\nu \ell})(a_{\nu \ell},t)_0. 
\end{equation}
Now we may use (\cite{Ka}, Lemma~4), Equation~\eqref{(1.6.2)} and Proposition~\ref{Proposition 2.5} to conclude that $$\prod_\ell\Delta_{f_\nu}(F)(a_{\nu \ell})=\prod_\ell\Delta(F,f_\nu)(a_{\nu \ell},t)_0. $$ \end{remark}

\begin{theorem}\label{Theorem 2.7} Suppose $\Delta_{A[[t]]}(F)\not=0$ and $F_0=\prod_\nu f_\nu^{n_\nu}$, where $f_\nu=\prod_l(x-a_{\nu \ell})$ for all $\nu$ with $a_{\nu \ell}\not=a_{\nu'\ell'}$ if $(\nu,\ell)\not=(\nu',\ell')$. Then
\begin{equation}\label{1.7.3}
\Delta_{A[[t]]}(F)_0=\prod_{\nu,\ell}\Delta(F,f_\nu)(a_{\nu \ell},t)_0\prod_\nu\Delta(f_\nu)^{n_\nu}\prod\res(f_\nu,f_{\nu'})^{2(n_\nu+n_{\nu'}-n_\nu n_{\nu'})}.
\end{equation}
\end{theorem}
\begin{proof} Consider the $t$-roots of $F$
\begin{equation}\label{1.7.4}
x=s_{\nu \ell i_\nu}=a_{\nu \ell} + \cdots,
\end{equation}
where $ 1\leq i_\nu \leq n_\nu$. The roots $s_{\nu li_\nu}$ are distincts by hypothesis and
\begin{align*}
\Delta_{A[[t]]}(F)_0 =  &\Big[\prod_{\nu,\ell}\Big(\prod(s_{\nu \ell i_\nu}-s_{\nu \ell j_\nu})^2\Big)_0\Big]\Big[\prod_\nu\Big(\prod(a_{\nu \ell'}-a_{\nu \ell})^{2n_\nu n_{\nu}}\Big)\Big]\prod(a_{\nu'\ell'}-a_{\nu \ell})^{2n_\nu n_{\nu'}}
           \\
           = & \Big[\prod_{\nu,\ell}\Big(\frac{\Delta(F,f_\nu)(a_{\nu \ell},t)_0}{\prod_{\ell' \not = \ell}(a_{\nu \ell}-a_{\nu \ell'})^{n_\nu(n_{\nu}-1)}\prod_{\nu'\not = \nu}(a_{\nu \ell}-a_{\nu'\ell'})^{(2n_\nu-2)n_{\nu'}}}\Big)\Big] 
           \\ 
            & \hspace{1cm} \Big[\prod_\nu\Big(\prod(a_{\nu \ell'}-a_{\nu \ell})^{2n_\nu n_{\nu}}\Big)\Big]\prod(a_{\nu'\ell'}-a_{\nu \ell})^{2n_\nu n_{\nu'}} \hspace{0.2cm} \textrm{ (by Equation~}\eqref{(1.5.2)})
           \\
          = & \prod_{\nu,\ell}\Delta(F,f_\nu)(a_{\nu \ell},t)_0\Big[\prod_\nu\Big(\prod(a_{\nu \ell}-a_{\nu \ell'})^{2n_\nu}\Big)\Big]\prod(a_{\nu'\ell'}-a_{\nu \ell})^{2(n_\nu+n_{\nu'}-n_\nu n_{\nu'})},
\end{align*}
as desired.\end{proof}

\begin{remark}{\label{Remark 2.8}} If $F(t)$ is a regular family then the second term of all the $t$-roots of $F(t)$ (obtained in the first truncation) are pairwise distinct. By another hand, the only hypothesis in Theorem~\ref{Theorem 2.7} is that the $t$-roots are pairwise distinct.
\end{remark}

\section{The vertices of a family of polynomials}

Let $F(x,t)=\sum_{i \geq 0} F_i(x)t^i$ be a family of polynomials of degree $n'$ with coefficients in $A$ and let $F_0=\prod_\nu f_\nu(x)^{n_\nu}$, where $f_\nu$ are coprime and monic polynomials as before.

\begin{definition}\label{Definition 3.1} Suppose $f_\nu=\prod_l(x-a_{\nu \ell})$ for all $\nu$. The \textit{vertex} of $F(t)$, denoted by $\ver(F)$, is defined by
$$
\ver(F):=\prod_{\nu,\ell}\Delta(F,f_\nu)(a_{\nu \ell},t)_0\prod\res(f_\nu,f_{\nu'})^{2(n_\nu+n_{\nu'}-n_\nu n_{\nu'})}.
$$
\end{definition}

\begin{proposition}\label{Proposition 3.2} Consider $F(x,t)$ the family of polynomials of the ``first kind", that is,
$$
F(x,t)=x^2A_{n-2}(x) + B_n(x)t+ O(t^2),
$$
where the subscripts denote the degree of an expression, $A_{n-2}(0)B_n(0) \not = 0$ and $A_{n-2}(x)$ is square-free. Then
$$
\ver(F)=-4B_n(0)A_{n-2}^3(0).
$$
\end{proposition}
\begin{proof} Indeed, first note that $F_1'=x^2A_{n-2}(0)+B_n(0)t$. Let $\alpha_1,\alpha_2$ be the two formal roots of $P_{F_1'}(s)=s^2A_{n-2}(0)+B_n(0)$. So
$$(\alpha_1 - \alpha_2)^2=\frac{-4B_n(0)}{A_{n-2}(0)}.$$
By Lemma~\ref{Lemma 2.4} we have 
$$\Delta(F,x)(0,t)_0=(A_{n-2}(0))^{2\cdot 2 - 2} \cdot\frac{-4B_n(0)}{A_{n-2}(0)} =-4A_{n-2}(0)B_n(0).$$
Now, by Definition~\ref{Definition 3.1}
$$\begin{aligned} \ver(F) 	& =  \Delta(F,x)(0,t)_0 \cdot 1 \cdot \res(A_{n-2}(x),x)^{2(2+1-2)}\\ & = -4A_{n-2}(0)B_n(0)A_{n-2}^2(0) \\ 
						& = -4B_n(0)A_{n-2}^3(0),\end{aligned} $$
as claimed.\end{proof}

\begin{proposition}\label{Proposition 3.3} Consider $F(x,t)$ the family of polynomials of the ``second kind", that is,
$$
F(x,t)=x^2A_{n-2}(x) + 2xB_{n-1}(x)t + C_n(x)t^2 + O(t^3),
$$
where $A_{n-2}(0)\big(B_{n-1}^2(0)-A_{n-2}(0)C_n(0)\big) \not = 0$ and $A_{n-2}(x)$ is square-free. Then
$$
\ver(F)=4\Big(B_{n-1}^2(0)-A_{n-2}(0)C_n(0)\Big)A_{n-2}^2(0).
$$
\end{proposition}
\begin{proof} Note that $F_1'=x^2A_{n-2}(0)+2xB_{n-1}(0)t+C_n(0)t^2$. Let $\alpha_1,\alpha_2$ be the two formal roots of $P_{F_1'}(s)=s^2A_{n-2}(0)+2sB_{n-1}(0)+C_n(0)$. So
$$(\alpha_1 - \alpha_2)^2=\frac{4B_{n-1}^2(0)-4A_{n-2}(0)C_n(0)}{A_{n-2}^2(0)}.$$
By Lemma~\ref{Lemma 2.4} we have
$$\begin{aligned} \Delta(F,x)(0,t)_0&=(A_{n-2}(0))^{2\cdot 2 - 2} \cdot\frac{4B_{n-1}^2(0)-4A_{n-2}(0)C_n(0)}{A_{n-2}^2(0)}\\ & =4\Big(B_{n-1}^2(0)-A_{n-2}(0)C_n(0)\Big).
\end{aligned}$$
Now, by Definition~\ref{Definition 3.1} 
$$\begin{aligned} \ver(F) & =  \Delta(F,x)(0,t)_0 \cdot 1 \cdot \res(A_{n-2}(x),x)^{2(2+1-2)}\\ &= 4\Big(B_{n-1}^2(0)-A_{n-2}(0)C_n(0)\Big)A_{n-2}^2(0), \end{aligned}$$
as claimed.\end{proof}

\begin{proposition}\label{Proposition 3.4} Consider $F(x,t)$ the family of polynomials of the ``third kind", that is,
$$
F(x,t)=x^2a_0+x^3a_1+2xb_0t+x^2b_1t+c_0t^2+xc_1t^2+d_1t^3+O(t^4),
$$
where $b_0^2-a_0c_0=0$ and $a_{0}(-a_1b_0^3 + b_1b_0^2a_0 - c_1b_0a_0^2 + d_1a_0^3) \not = 0$. Then
$$
\ver(F)=-4(-a_1b_0^3 + b_1b_0^2a_0 - c_1b_0a_0^2 + d_1a_0^3).
$$
\end{proposition}
\begin{proof} Note that $F_1'=x^2a_0+2xb_0t+c_0t^2$. Since $b_0^2-a_0c_0=0$, let
$\alpha=-b_0/a_0$ be the only root of $P_{F_1'}(s)=s^2a_0+2sb_0+c_0$. In this
case, we need compute more terms of the $t$-roots of $F$. Consider the following
substitution on $F$, 
$$x=t(x_1-b_0/a_0),$$
where $x_1$ is a new variable. So $F(x,t)=t^2G(x_1,t)$, where 
$$G(x_1,t)=x_1^2a_0 + \Big[a_1\Big(\frac{-b_0}{a_0}\Big)^3 + b_1\Big(\frac{-b_0}{a_0}\Big)^2+c_1\Big(\frac{-b_0}{a_0}\Big)+d_1\Big]t +O(t^2).$$
The first truncation of $G$ is 
$$G_1'=x_1^2a_0 + \Big[a_1\Big(\frac{-b_0}{a_0}\Big)^3 + b_1\Big(\frac{-b_0}{a_0}\Big)^2+c_1\Big(\frac{-b_0}{a_0}\Big)+d_1\Big]t.$$
Since $a_{0}(-a_1b_0^3 + b_1b_0^2a_0 - c_1b_0a_0^2 + d_1a_0^3) \not = 0$, let $\alpha_1,\alpha_2$ be the two roots of 
$$P_{G_1'}(s)=s^2a_0 + \Big[a_1\Big(\frac{-b_0}{a_0}\Big)^3 + b_1\Big(\frac{-b_0}{a_0}\Big)^2+c_1\Big(\frac{-b_0}{a_0}\Big)+d_1\Big].$$
So $$ (\alpha_1-\alpha_2)^2= \frac{-4a_0\cdot \Big[a_1\Big(\frac{-b_0}{a_0}\Big)^3 + b_1\Big(\frac{-b_0}{a_0}\Big)^2+c_1\Big(\frac{-b_0}{a_0}\Big)+d_1\Big]}{a_0^2}.$$
By Lemma~\ref{Lemma 2.4} we have that 
$$\begin{aligned}
	\Delta(F,x)(0,t)_0 & =(0\cdot a_1+a_0)^{2\cdot 2 -2}\cdot 1 \cdot \frac{-4a_0\cdot \Big[a_1\Big(\frac{-b_0}{a_0}\Big)^3 + b_1\Big(\frac{-b_0}{a_0}\Big)^2+c_1\Big(\frac{-b_0}{a_0}\Big)+d_1\Big]}{a_0^2} \\ & = -4a_0\Big[a_1\Big(\frac{-b_0}{a_0}\Big)^3 + b_1\Big(\frac{-b_0}{a_0}\Big)^2+c_1\Big(\frac{-b_0}{a_0}\Big)+d_1 \Big].\end{aligned}$$
Now, by Definition~\ref{Definition 3.1}
$$\begin{aligned} \ver(F) & =  \Delta(F,x)(0,t)_0 \cdot 1 \cdot \res(a_1x+a_0,x)^{2(2+1-2)} \\
					  & = -4a_0\Big[a_1\Big(\frac{-b_0}{a_0}\Big)^3 + b_1\Big(\frac{-b_0}{a_0}\Big)^2+c_1\Big(\frac{-b_0}{a_0}\Big)+d_1 \Big]\cdot a_0^2 \\ 
					  & = -4(-a_1b_0^3 + b_1b_0^2a_0 - c_1b_0a_0^2 + d_1a_0^3),\end{aligned}$$
as claimed.\end{proof}

\begin{definition}\label{Definition 3.5} 
Let $n\in \mathbb{Z}_{\geq 1}$ and  $\alpha, \delta_n \in A[x]$. We say that $F(x,t)$ satisfies $(\alpha,\delta_n)$-Zd (Zeuthen decomposition) for $f_\nu$ if $n_\nu >1$, $\gcd(\alpha \delta_n,f_\nu)=1$ and $$\alpha F(t)=(f_\nu\beta_0+\beta_1t+\cdots +\beta_{n-1}t^{n-1})^{n_\nu} + \delta_n t^n + O(t^{n+1}) $$ with $\beta_i \in A[x]$ for all $i=0,...,n-1$. \end{definition}

\begin{lemma}\label{Lemma 3.6} Suppose that $F(x,t)$ satisfies $(\alpha,\delta_n)$-Zd for $f_\nu=x$. Then
$$\Delta(F,x)(0,t)_0=\frac{(-1)^\frac{n_\nu(n_\nu-1)}{2}n_\nu^{n_\nu}\cdot \delta_n(0)^{n_\nu-1}\cdot\Big(\prod_{\nu' \not = \nu}f_{\nu'}^{n_{\nu'}}(0)\Big)^{n_\nu-1}}{\alpha(0)^{n_\nu-1}}.$$
\end{lemma}

\begin{proof}  First, note that $\alpha(0)\delta_n(0)\not = 0$. Now, let $\beta_i \in A[x]$ with $i=0,...,n-1$ such that 
$$G:=\alpha F=(x\beta_0+\beta_1t+\cdots + \beta_{n-1}t^{n-1})^{n_\nu}+\delta_nt^n + O(t^{n+1}).$$
Consider $i_1:=\min\{ i \ | \ i>0 \textrm{ and } \beta_i(0) \not = 0  \} $. If $i_1\cdot n_\nu < n$, then the first truncaction of $G$ is
$$G_1'=(x\beta_0(0)+\beta_{i_1}(0)t^{i_1})^{n_\nu}. $$
In this case, $P_{G_1'}(s)=(s\beta_0(0)+\beta_{i_1}(0))^{n_\nu}$. Substituting $x=t^{i_1}(x_1+\phi_1)$ on $G$, where $\phi_1$ is the only root of $P_{G_1'}(s)$ and $x_1$ is a new variable. After a finite number of steps, using recurrence, we may assume that $G=t^mH$ where
$$H=\Big(x_j\big(\beta_0(0)+x_j\beta_0'\big)+\beta_1't+\cdots + \beta_{n'-1}'t^{n'-1}\Big)^{n_\nu}+\big(\delta_n(0)+x_j\delta_{n'}'\big)t^{n'} + O(t^{n'+1}) $$
with $x_j$ a variable and $i_j:=\min\{ i \ | \ i>0 \textrm{ and } \beta_i'(0) \not = 0 \} $  satisfies $i_j\cdot n_\nu\geq n'$. 

\underline{If $i_j\cdot n_\nu=n'$}. In this case, the first truncation of $H$ is
$$ H_1'=\big(x_j\beta_0(0)+\beta_{i_j}'(0)t^{i_j}\big)^{n_\nu}+\delta_n(0)t^{n'} $$
and $P_{H_1'}(s)=\big(s\beta_0(0)+\beta_{i_j}'(0)\big)^{n_\nu}+\delta_n(0)$. If $\alpha_1,...,\alpha_n$ are the roots of $P_{H_1'}(s)$ then 
\begin{equation}\label{(2.6.6)}
\prod_{i<j}(\alpha_i-\alpha_j)^2 = \frac{(-1)^{\frac{n_\nu(n_\nu-1)}{2}}n_\nu^{n_\nu}\cdot\delta_n(0)^{n_\nu-1}\cdot\beta_0(0)^{n_\nu(n_\nu-1)}}{(\beta_0(0)^{n_\nu})^{2n_\nu-2}}.
\end{equation}

\underline{If $i_j\cdot n_\nu>n'$}. In this case, the first truncation of $H$ is
$$ H_1'=\big(x_j\beta_0(0)\big)^{n_\nu}+\delta_n(0)t^{n'} $$
and $P_{H_1'}(s)=\big(s\beta_0(0)\big)^{n_\nu}+\delta_n(0)$. If $\alpha_1',...,\alpha_n'$ are the roots of $P_{H_1'}(s)$ then

\begin{equation}\label{(2.6.7)}
\prod_{i<j}(\alpha_i'-\alpha_j')^2 = \frac{(-1)^{\frac{n_\nu(n_\nu-1)}{2}}n_\nu^{n_\nu}\cdot\delta_n(0)^{n_\nu-1}\cdot\beta_0(0)^{n_\nu(n_\nu-1)}}{(\beta_0(0)^{n_\nu})^{2n_\nu-2}}.
\end{equation}
By  Definition~\ref{Definition 3.1}, Lemma~\ref{Lemma 2.4}, Equation~\eqref{(2.6.6)} and Equation~\eqref{(2.6.7)} we have that 
$$\alpha(0)^{2n_\nu-2}\cdot\Delta(F,x)(0,t)_0  = \Delta(G,x)(0,t)_0=(-1)^{\frac{n_\nu(n_\nu-1)}{2}}n_\nu^{n_\nu}\cdot\delta_n(0)^{n_\nu-1}\cdot\beta_0(0)^{n_\nu(n_\nu-1)},$$
that is, 
$$\Delta(F,x)(0,t)_0  =\frac{(-1)^\frac{n_\nu(n_\nu-1)}{2}n_\nu^{n_\nu}\cdot \delta_n(0)^{n_\nu-1}\cdot\beta_0(0)^{n_k(n_\nu-1)}}{\alpha(0)^{2n_\nu-2}}.$$
Since
$$\alpha \cdot \prod_{\nu' \not = \nu}f_{\nu'}^{n_{\nu'}}=\beta_0^{n_{\nu}}$$
then 
$$
\Delta(F,x)(0,t)_0 = \frac{(-1)^\frac{n_\nu(n_\nu-1)}{2}n_\nu^{n_\nu}\cdot \delta_n(0)^{n_\nu-1}\cdot\Big(\prod_{\nu' \not = \nu}f_{\nu'}^{n_{\nu'}}(0)\Big)^{n_\nu-1}}{\alpha(0)^{n_\nu-1}}, $$
as desired. \end{proof}

\begin{theorem}\label{Theorem 3.7} Suppose $F(x,t)$ satisfies $(\alpha_\nu,\delta_{\nu,n(\nu)})$-Zd for each $f_\nu$ with $n_\nu>1$ and $f_\nu=\prod_{\ell=1}^{n_\nu}(x-a_{\nu \ell})$ with all $a_{\nu \ell}$ distincts.  Then  
$$\ver(F)=\frac{c \prod_\nu \res\big(f_\nu,\delta_{\nu,n(\nu)}^{n_\nu-1}\big)\Big[\prod_{\nu}\Big(\prod_{\nu'\not=\nu}\res(f_{\nu},f_{\nu'})^{n_{\nu'}}\Big)\Big]}{\prod_\nu\res(f_\nu,\alpha_\nu^{n_\nu-1})},$$ where $c \in k - \{0\}$. \end{theorem} 

\begin{proof} First we note that $\alpha_\nu(a_{\nu \ell})\delta_{\nu,n(\nu)}(a_{\nu \ell})\not = 0$ for all $ \ell=1,...,n_\nu$. Now, by Lemma~\ref{Lemma 2.4} we have the following equality $$\Delta\big(F(x,t),x-a_{\nu \ell}\big)(a_{\nu \ell},t)_0 = \Delta\big(F(x+a_{\nu \ell},t),x\big)(0,t)_0.$$
By Lemma~\ref{Lemma 3.6} we have that
$$\Delta\big(F(x+a_{\nu \ell}),x\big)(0,t)_0=\frac{c_{\nu \ell}\cdot\delta_\nu(a_{\nu \ell})^{n_k-1}\big[\prod_{\nu '\not = \nu }f_{\nu '}(a_{\nu \ell})^{n_{\nu '}} \big]^{n_\nu-1} \big[\prod_{\ell'\not =\ell}(a_{\nu \ell'}-a_{\nu \ell})^{n_\nu }\big]^{n_\nu-1}}{\alpha_\nu (a_{\nu \ell})^{n_\nu -1}},$$
where $c_{\nu \ell} \in k-\{0\}$. By both above equations and Proposition~\ref{Proposition 2.5} we conclude

$$\prod_{\nu ,\ell}\Delta(F,f_\nu )(a_{\nu \ell},t)_0=\frac{c \prod_\nu  \res\big(f_\nu ,\delta_{\nu ,n(\nu )}^{n_\nu -1}\big)\Big[\prod_{\nu }\Big(\prod_{\nu '\not=\nu }\res(f_{\nu },f_{\nu '})^{n_{\nu '}(n_\nu -1)}\Big)\Big]}{\prod_\nu \res(f_\nu ,\alpha_\nu ^{n_\nu -1})},$$
where $c\in k-\{0\}$. Now, by Definition~\ref{Definition 3.1}
$$\begin{aligned}
\ver(F) & = \prod_{\nu,\ell}\Delta(F,f_\nu)(a_{\nu \ell},t)_0 \prod \res(f_\nu,f_{\nu'})^{2(n_\nu+n_{\nu'}-n_\nu n_{\nu'})}\\
	 & =  \frac{c \prod_\nu  \res\big(f_\nu ,\delta_{\nu ,n(\nu )}^{n_\nu -1}\big)\prod\res(f_{\nu },f_{\nu '})^{2n_{\nu '}n_\nu -n_\nu-n_{\nu'}}}{\prod_\nu \res(f_\nu ,\alpha_\nu ^{n_\nu -1})} \cdot \prod \res(f_\nu,f_{\nu'})^{2(n_\nu+n_{\nu'}-n_\nu n_{\nu'})} \\ 
	 & = \frac{c \prod_\nu  \res\big(f_\nu ,\delta_{\nu ,n(\nu )}^{n_\nu -1}\big)\prod\res(f_{\nu },f_{\nu '})^{n_\nu+n_{\nu'}}}{\prod_\nu \res(f_\nu ,\alpha_\nu ^{n_\nu -1})} \\ 
	 & = \frac{c \prod_\nu \res\big(f_\nu,\delta_{\nu,n(\nu)}^{n_\nu-1}\big)\Big[\prod_{\nu}\Big(\prod_{\nu'\not=\nu}\res(f_{\nu},f_{\nu'})^{n_{\nu'}}\Big)\Big]}{\prod_\nu\res(f_\nu,\alpha_\nu^{n_\nu-1})},
\end{aligned}$$
as desired.\end{proof}

\begin{corollary}\label{Corollary 3.8} Suppose $F(x,t)$ satisfies $(\alpha_\nu,\delta_{\nu,n(\nu)})$-Zd for each $f_\nu$ with $n_\nu>1$ and $f_\nu=\prod_{\ell=1}^{n_\nu}(x-a_{\nu \ell})$ with all $a_{\nu \ell}$ distincts. Then  $$\Delta_{A[[t]]}(F)_0=\prod_\nu \Delta(f_\nu )^{n_\nu } \cdot \frac{c \prod_\nu  \res\big(f_\nu ,\delta_{\nu ,n(\nu )}^{n_\nu -1}\big)\Big[\prod_{\nu }\Big(\prod_{\nu '\not=\nu }\res(f_{\nu },f_{\nu '})^{n_{\nu '}}\Big)\Big]}{\prod_\nu \res(f_\nu ,\alpha_\nu ^{n_\nu -1})},$$ where $c \in k - \{0\}$. \end{corollary}

\begin{proof} By Theorem~\ref{Theorem 2.7} and Definition~\ref{Definition 3.1} we have that $$\Delta_{A[[t]]}(F)_0= \prod_\nu \Delta(f_\nu )^{n_\nu } \cdot \ver(F).$$ Now, the desired result follows by using Theorem~\ref{Theorem 3.7}. \end{proof}

\begin{remark}{\label{Remark 3.9}} By Remark~\ref{Remark 2.8} a family of polynomials of the third kind is not a regular family. A family of polynomials satisfying the hypothesis of Theorem~\ref{Theorem 3.7} is not a regular family in general.\end{remark}

\section{The general type of a family of polynomials}

Let $F(t)=\sum_{i \geq 0} F_i(x)t^i,$ be a family of polynomials of degree $n'$ with coefficients in $A$, a domain and $k$-algebra with $k$ a algebraically closed field of characteristic zero. Let $F_0=\prod_\nu f_\nu(x)^{n_\nu}$, where $f_\nu$ are coprime and monic polynomials. Let $s_1(t),...,s_{n'}(t)$ be the $t$-roots of $F$ as before. The constant term of $s_i$ will be denoted by $s_i(0).$

\begin{definition}\label{Definition 4.1} Let $m(i,j)$ be the order of $s_i-s_j$, that is, $$s_i-s_j=(s_i-s_j)_0t^{m(i,j)} + \cdots $$ where $(s_i-s_j)_0\not = 0$ is the leading coefficient of $s_i-s_j$.\end{definition}

\begin{definition}\label{Definition 4.2} Suppose $\Delta_{A[[t]]}(F)\not=0$. The \textit{type} $\tau_\nu$ of $F(t)$ with respect to $f_\nu$ is $$\tau_\nu=\lfloor n_\nu\cdot m_\nu\rfloor,$$ where $ m_\nu=\max\big\{m(i,j)\ | \ f_\nu(s_i(0))=0 \ \textrm{ and } \ f_\nu(s_j(0))=0 \big\}$. \end{definition}
 
\begin{proposition}\label{Proposition 4.3} Consider the family of polynomials  $$ F(t)=x^2A_{n-2}(x) + B_n(x)t+ O(t^2) $$ of the first kind.  The type of $F(t)$ for $x$ is $\tau=1$.   \end{proposition}

\begin{proof} If $A_{n-2}(x)=\prod_{i}(x-a_i)$, with $a_i\not = 0$ for all $i$ and $a_i\not = a_j$ for all $i\not = j$, then the $t$-roots of $F(t)$ are the following $$s_1=0+\alpha_1t^{1/2}+\cdots, s_2=0+\alpha_2t^{1/2}+\cdots,s_{3}=a_1+\cdots ,..., s_{i+2}=a_i+\cdots $$ where $\alpha_1,\alpha_2$ are the roots of $P_{F_1'}(s)=s^2A_{n-2}(0)+B_n(0)$. This implies that $t$ the type of $F$ for $x$ is $$\tau=\lfloor2\cdot m \rfloor,$$ where $$m=\max\big\{m(i,j)\ | \ x(s_i(0))=0 \ \textrm{ and } \ x(s_j(0))=0 \big\}=m(1,2)=1/2. $$ So $\tau=1$.\end{proof}

\begin{proposition}\label{Proposition 4.4} Consider the family of polynomials $$ F(t)=x^2A_{n-2}(x) + 2xB_{n-1}(x)t + C_n(x)t^2 + O(t^3)$$ of the second kind. The type of $F(t)$ for $x$ is $\tau=2$. \end{proposition}

\begin{proof} If $A_{n-2}(x)=\prod_{i}(x-a_i)$, with $a_i\not = 0$ for all $i$ and $a_i\not = a_j$ for all $i\not = j$, then the $t$-roots of $F(t)$ are the following $$s_1=0+\alpha_1t^{1}+\cdots, s_2=0+\alpha_2t^{1}+\cdots,s_{3}=a_1+\cdots ,..., s_{i+2}=a_i+\cdots $$ where $\alpha_1,\alpha_2$ are the roots of $P_{F_1'}(s)=s^2A_{n-2}(0)+2B_{n-1}(0)+C_n(0)$. This implies that $\tau$ the type of $F$ for $x$ is $$\tau=\lfloor2\cdot m \rfloor,$$ where $$m=\max\big\{m(i,j)\ | \ x(s_i(0))=0 \ \textrm{ and } \ x(s_j(0))=0 \big\}=m(1,2)=1. $$ So $\tau=2$.\end{proof}

\begin{proposition}\label{Proposition 4.5} Consider the family of polynomials$$ F(t)=a_0x^2+a_1x^3+2b_0xt+b_1x^2t+c_0t^2+c_1xt^2+d_1t^3+O(t^4)$$ of the third kind. The type of $F(t)$ for $x$ is $\tau=3$. \end{proposition}

\begin{proof} Since $F_0(x)=a_1x^2(x+a_0/a_1)$, then the $t$-roots of $F(t)$ are the following $$s_1=0+\alpha t^{1}+\alpha_1t^{1+1/2}+\cdots, s_2=0+\alpha t^{1}+\alpha_2t^{1+1/2}+\cdots,s_{3}=-a_0/a_1+\cdots  $$ where $\alpha$ is the only root of $P_{F_1'}(s)=a_0s^2+2b_0s+c_0$ and $\alpha_1,\alpha_2$ are the two distincts roots of $$P_{G_1'}(s)=s^2a_0+a_1\Big(\frac{-b_0}{a_0}\Big)^3+b_1\Big(\frac{-b_0}{a_0}\Big)^2+c_1\Big(\frac{-b_0}{a_0}\Big)+d_1.$$ This implies that $\tau$ the type of $F$ for $x$ is $$\tau=\lfloor2\cdot m \rfloor,$$ where $$m=\max\big\{m(i,j)\ | \ x(s_i(0))=0 \ \textrm{ and } \ x(s_j(0))=0 \big\}=m(1,2)=3/2. $$ So $\tau=3$.\end{proof}

\begin{proposition}\label{Proposition 4.6} Let $F(t)$ be a family of polinomials that satisfies $(\alpha,\delta_{n})$-Zd for $f_\nu=(x-a_{\nu})$. The type of $F(t)$ for $f_\nu$ is $\tau_\nu=n$.  \end{proposition}

\begin{proof}  We may assume, without loss of generality, $a_{\nu}=0$. Define $G(t)=\alpha F(t)$. By hypothesis, $$G=\Big(x\beta_{0}+\beta_1t+\cdots +\beta_{n-1}t^{n-1}\Big)^{n_\nu}+\delta_{n}t^n+O(t^{n+1}).$$ By the proof of Lemma~\ref{Lemma 3.6} we have that $$s_k=t^{i_1/n_\nu}(\phi_1+t^{i_2/n_\nu}(\phi_2+\cdots +(\phi_{j}+t^{(n-i_1-i_2-\cdots-i_j)/n_\nu}(\alpha_{k}+\cdots $$ for $k=1,...,n_\nu$ and $f_\nu(s_k(0))\not = 0$ for $k>n_\nu$. This implies that for $1\leq k,k' \leq n_\nu$ with $k\not = k'$ we have that $$s_k-s_{k'}=(\alpha_{k}-\alpha_{k'})t^{n/n_\nu}+\cdots $$ So $$m_\nu=\max \big\{m(k,k') \ | \ f_\nu(s_k(0))=0 \ \textrm{ and } f_\nu(s_{k'}(0))=0 \big\}= n/n_\nu $$ and $\tau_\nu =n$.\end{proof}

\begin{theorem}\label{Theorem 4.7} Let $F(t)$ be a family of polynomials such that $\Delta_{A[[t]]}(F)\not=0$ and let $\tau_\nu$ be the type of $F(t)$ for $f_\nu$ and consider $\tau:=\max_\nu\{\tau_\nu\}$. The vertex $\ver(F)$ depends only on the first $\tau$ terms of the series $F(t)$.\end{theorem}

\begin{proof} First we note that if $$s_j=\alpha_{\nu,0}+ \alpha_{\nu,1}t^{\nu_1}+\cdots +\alpha_{\nu,k}t^{\nu_k}+\cdots$$ is a $t$-root of $F$ such that $f_\nu(s_j(0))=0$, so we need to consider up to the first $\floor{n_\nu \cdot \nu_k}$ terms of $F(t)$ to find $\alpha_{\nu,k}$. This implies that we need to consider up to the first $\tau_\nu$ terms of $F(t)$ to obtain $(s_k-s_{k'})_0$ for $1\leq k,k'\leq n_\nu$ with $k\not= k'$, where $f_\nu(s_k(0))=0$ and $f_\nu(s_{k'}(0))=0$. By Lemma~\ref{Lemma 2.4} and Proposition~\ref{Proposition 2.5} we may conclude that we need to consider only the first $\tau_\nu$ terms to have $\Delta(F,f_\nu)(a_{\nu,l},t)_0$. Hence, by Definition~\ref{Definition 3.1} we need to consider only the first $\tau$ terms of $F(t)$ to compute $\ver(F)$. \end{proof}

\begin{definition}\label{Definition 4.8} Let $F(t)$ be a family of polynomials such that $\Delta_{A[[t]]}(F(t))\not=0$. The number $\tau$ in the Theorem~\ref{Theorem 4.7} is called the \textit{general type} of $F(t)$.\end{definition}

\begin{example}\label{Example 4.9} Consider the following regular families (\cite{Ka}), where $A:=k[y]$: \begin{enumerate} \item[$(a)$] $F(t)=x^3+F_1t+\cdots$, where $x\nmid F_1$. The general type of $F(t)$ is $\tau =1$; \item[$(b)$] $F(t)=x^3+xC_1t+\cdots $, where $x\nmid C_1$. The general type of $F(t)$ is $1$; \item[$(c)$] $F(t)=x^3+xC_2t^2+ F_3t^3+\cdots $, where $x\nmid (4C_2^3+27F_3^2)$. The general type of $F(t)$ is $3$; \item[$(d)$] $F(t)=x^2y+F_1t+\cdots $, where $x\nmid F_1$. The general type of $F(t)$ is $1$; \item[$(e)$] $F(t)=x^2y+xC_1t+F_2t^2+\cdots $, where $x\nmid(C_1^2-4F_2y)$. The general type of $F(t)$ is $2$.\end{enumerate} \end{example}

\section*{Acknowledgement} The autor would like to thank E. Esteves, S. Kleiman, N. Medeiros and I. Vainsencher  for many discussions on the subject. The autor would like to thank IMPA for its hospitality while this work was being started.

\vspace{1cm}

\textsc{Universidade Federal da Para\'iba, Centro de Ci\^encias Exatas e da Natureza, Departamento de Matem\'atica, Campus Universit\'ario, 58051-900 Jo\~aoo Pessoa PB, Brazil}

\textit{E-mail address:} \texttt{wallace@mat.ufpb.br}

\end{document}